\documentclass[twoside,11pt,leqno]{article}

\usepackage{itjournal2}
\usepackage{amssymb}
\usepackage{amsmath}
\usepackage{amsfonts}
\usepackage{enumerate}
\usepackage{amsthm}
\usepackage{indentfirst}
\usepackage{ot1patch}
\usepackage{psfrag}
\usepackage{mathtools}
\DeclareMathOperator{\tr}{tr}

\usepackage{amsthm}

\newtheorem{Theorem}{Theorem}[section]

\newtheorem{proposition}{Proposition}[section]
\newtheorem{corollary}{Corollary}[section]

\let\oldproofname=\proofname
\renewcommand{\proofname}{\rm\bf{\oldproofname}}

\usepackage[all]{xy}

\itjournalheading{N.}{...--}{... (1--8)}

\ShortHeadings{article title}{name of author(s)}
\firstpageno{1}

\begin{document}

\title{Matrices of infinite dimensions and their applications}

\author{\name \L ukasz Matysiak\\
	\addr ul. Powsta\'{n}c\'{o}w Wielkopolskich 2, 85-090 Bydgoszcz, Poland\\
	Kazimierz Wielki University\\
	\email{lukmat@ukw.edu.pl} \AND 
	\name Weronika Przewo\'zniak\\
	\addr ul. Powsta\'{n}c\'{o}w Wielkopolskich 2, 85-090 Bydgoszcz, Poland\\
	Kazimierz Wielki University\\
	\email{weronika.przewozniak@student.ukw.edu.pl}
	\AND 
	\name Natalia Ruli\'nska\\
	\addr ul. Powsta\'{n}c\'{o}w Wielkopolskich 2, 85-090 Bydgoszcz, Poland\\
	Kazimierz Wielki University\\
	\email{natalia.rulinska@student.ukw.edu.pl}
}


\maketitle

\begin{abstract}
Matrices are very popular and widely used in mathematics and other fields of science. Every mathematician has known the properties of finite-sized matrices since the time of study. In this paper, we consider the basic theory of infinite matrices. So far, there have been references and few results in certain scientific fields, but they have not been thoroughly researched.
\end{abstract}


\section{Introduction}

Infinite matrices, the forerunner and a main constituent of many branches of classical mathematics (infinite quadratic forms, integral equations, differential equations, etc.) and of the modern operator theory, is revisited to demonstrate its deep influence on thedevelopment of many branches of mathematics, classical and modern, replete with applications.

\medskip

It is known that we can add matrices, multiply by scalar, multiply matrices, calculate determinant, calculate inverse matrix, determine rank matrix. We can find all these properties of matrices in many basic academic books, for example in \cite{1}, \cite{2}, \cite{3}.

\medskip

Matrices with an infinite number of rows and / or columns are also considered - formally, it is sufficient that for any elements indexing rows and columns there is a well-defined matrix element (index sets do not even have to be subsets of natural numbers). Similarly to the finite case, we can define addition, subtraction, multiplication by scalar or matrix shifting, although matrix multiplication requires some assumptions.

\medskip

Applications of matrices are found in most scientific fields (\cite{5}). In every branch of physics, including classical mechanics, optics, electromagnetism, quantum mechanics, and quantum electrodynamics, they are used to study physical phenomena, such as the motion of rigid bodies.
In computer graphics, they are used to manipulate 3D models and project them onto a 2-dimensional screen. In probability theory and statistics, stochastic matrices are used to describe sets of probabilities. For example, they are used within the PageRank algorithm that ranks the pages in a Google search. (\cite{4}) Matrix calculus generalizes classical analytical notions such as derivatives and exponentials to higher dimensions. Matrices are used in economics to describe systems of economic relationships.

\medskip

A major branch of numerical analysis is devoted to the development of efficient algorithms for matrix computations, a subject that is centuries old and is today an expanding area of research. Matrix decomposition methods simplify computations, both theoretically and practically. Algorithms that are tailored to particular matrix structures, such as sparse matrices and near-diagonal matrices, expedite computations in finite element method and other computations. Infinite matrices occur in planetary theory and in atomic theory. A simple example of an infinite matrix is the matrix representing the derivative operator, which acts on the Taylor series of a function. 

\medskip

In this paper, we formalize and develop the basic theory of infinite matrices. So far, they have not been thoroughly researched, despite their significant use in some fields of science.

\section{Results}

By an infinite dimension matrix we call a matrix for which the number of rows is infinite or the number of columns is infinite.

\medskip

We define zero, triangular, diagonal, unitary and transposed matrices of an infinite dimension very analogously. 

\medskip

A square matrix of an infinite dimension is a matrix in which the number of rows is equinumerous to the number of columns. 

\medskip

Matrix sum and by scalar multiplication are also analogous.

\medskip

\begin{corollary} 
	If we try to multiply matrix $A_{m\times n}$ with matrix $B_{n\times k}$, we get the following conclusions:
	\begin{itemize}
		\item[(a) ] If $m=\infty$, $k=\infty$, then $AB=C_{\infty\times\infty}$.
		\item[(b) ] If $n=\infty$, then $AB=C_{m\times k}=[c_{ij}]$, where $c_{ij}=\sum_{l=1}^{\infty}a_{il}b_{lj}$ ($1\leqslant i\leqslant m$, $1\leqslant j\leqslant k$) be a convergent series.
		\item[(c) ] If $A$, $B$ be square matrices of an infinite dimension, then $AB=C$ holds.
		\item[(d) ] If $A$, $B$ be matrices of an infinite dimension and the number of rows in $A$ is equinumerous to the number of columns in $B$, then we can multiply matrices $A$ and $B$ only if rows of $A$ and columns of $B$ be a convergent series.
	\end{itemize}
\end{corollary} 

Let $M_1(\infty, R)=M_1(R)$ be denote the set of all square matrices of an infinite dimension with coefficients from any integral domain $R$, where all rows and columns are convergent series. Then $M_1(\infty, R)$ be a ring. 
Easy to check that $\{A\in M_1(\infty,\mathbb{Z})\colon \det A\in\{-1,1\}\}$ and $\{A\in M_1(\infty,\mathbb{Z})\colon \det A=1\}$ are multiplicative groups.

\medskip

The determinant of a square matrix $A$ of finite dimension can be easily determined by the formula:
$$\det A=\det(\exp(\log A))=\exp(\tr(\log A)),$$
where $\log A=\sum_{k=1}^{\infty}(-1)^{k+1}\dfrac{A^k}{k}$.
For an infinite dimension we must add the assumption that $\tr(\log A)$ be a convergent series.

\begin{proposition}
	Let $A$ be an $m\times n$ matrix, and let $B$ be an matrix $n\times m$, where $m, n\in\mathbb{N}\cup\{\infty\}$. Let $1\leqslant j_1, j_2, \dots, j_m\leqslant n$.
	Let $A_{j_1j_2\dots j_m}$ denote the $m\times m$ matrix consisting of columns $j_1, j_2, \dots, j_m$ of $A$. Let $B_{j_1j_2\dots j_m}$ denote the $m\times m$ matrix consisting of rows $j_1, j_2, \dots, j_m$ of $B$. 
	Then
	$$\det (AB)=\sum_{1\leqslant j_1<j_2<\dots <j_m\leqslant n}\det(A_{j_1j_2\dots j_m})\det(B_{j_1j_2\dots j_m}).$$
	
\end{proposition}

\begin{proof}
	First we will show the proof in the finite version.
	
	\medskip
	
	Let $(k_1, k_2, \dots, k_m)$ be an ordered $m$-tuple of integers.
	Let $\eta (k_1, k_2, \dots, k_m)$ denote the sign of $(k_1, k_2, \dots, k_m)$.
	Let $(l_1, l_2, \dots, l_m)$ be the same as $(k_1, k_2, \dots, k_m)$ except for $k_i$ and $k_j$ having been transposed.
	Then from Transposistion is of Odd Parity:
	$$\eta(l_1, l_2, \dots, l_m)=-\eta(k_1, k_2, \dots, k_m).$$
	Let $(j_1, j_2, \dots, j_m)$ be the same as $(k_1, k_2, \dots, k_m)$ by arranged into non-decreasing order.
	That is $j_1\leqslant j_2\leqslant\dots\leqslant j_m$.
	Then it follows that:
	$$\det(B_{k_1\dots k_m})=\eta(k_1, k_2, \dots, k_m)\det(B_{j_1\dots j_m}).$$
	Hence:
	\begin{align*} 
	\det(AB)=&\sum_{1\leqslant l_1, \dots, l_m\leqslant m}\eta(l_1, \dots, l_m)(\sum_{k=1}^{n}a_{1k}b_{kl_1})\dots (\sum_{k=1}^{n}a_{mk}b_{kl_m})=\\
	=&\sum_{1\leqslant k_1, \dots, k_m\leqslant n}a_{1k_1}\dots a_{mk_m}\sum_{1\leqslant l_1, \dots, l_m\leqslant m}\eta(l_1, \dots, l_m)b_{k_1l_1}\dots b_{k_ml_m}=\\
	=&\sum_{1\leqslant k_1, \dots, k_m\leqslant n}a_{1k_1}\dots a_{mk_m}\det(B_{k_1\dots k_m})=\\
	=&\sum_{1\leqslant k_1, \dots, k_m\leqslant n}a_{1k_1}\eta(k_1, \dots, k_m)\dots a_{mk_m}\det(B_{j_1\dots j_m})=\\
	=&\sum_{1\leqslant j_1\leqslant j_2\leqslant\dots\leqslant j_m\leqslant n}\det(A_{j_1\dots j_m})\det(B_{j_1\dots j_m}).
	\end{align*} 
	If two $j$s are equal:
	$$\det(A_{j_1\dots j_m})=0.$$
	
	\medskip
	
	For an infinite matrices we put $\infty$-tuple in proof in the form $(k_1, k_2, k_3, \dots)$.
	And put $1\leqslant j_1, j_2, j_3, \dots <n=\infty$.
\end{proof}

\begin{corollary}
	If $m=n$ ($m, n\in\mathbb{N}\cup\{\infty\}$), then 
	$$\det(AB)=\det(A)\det(B).$$
\end{corollary}

The following two Propositions give us a way to compute the inverse matrix.

\begin{proposition}
	Let $A$ be a matrix in which every rows and colums form convergent series such that $||I-A||<1$, where $||\cdot||$ is a submultiplicative norm. Then
	$$A^{-1}=I+(I-B)+(I-B)^2+\dots$$
\end{proposition}

\begin{proof}
	A matrix $A\in M_1(K)$ (where $n\in\mathbb{N}\cup\{\infty\}$, $K$ be a field) is invertible if and only if the map $f\colon K^n\to K^n$ defined by $f(x)=Ax$ is invertible, where elements of $K^n$ are considered as column vectors.
\end{proof}

\section{Applications}

Let $A$ be an $m\times n$ matrix over an arbitrary field $F$ ($m, n\in\mathbb{N}\cup\{\infty\}$). There is an associated linear mapping $f\colon F^n\to F^m$ defined by $f(x)=Ax$. The rank of $A$ is the dimension of the image $f$. This definition has the advantage that it can be applied to any linear map without need for a specific matrix. 

\medskip

Let
$$\left\{\begin{array}{ccccccc}
a_{11}x_1+a_{12}x_2+a_{13}x_3+\ldots &= b_1\\
a_{21}x_1+a_{22}x_2+a_{23}x_3+\ldots &= b_2\\
a_{31}x_1+a_{32}x_2+a_{33}x_3+\ldots &= b_3\\
\dots 
\end{array}\right.$$ ($AX=B$) be a system of equations.
By Cramer's system we mean a system in which the number of equations is equinumerous to the number of unknowns. Then Cramer's theorem states that in finite case the system has a unique solution provided we have $n$ equations. Hence individual values for the unknowns are given by:
$$x_i=\dfrac{\det A_i}{\det A},$$ 
for $i=1, 2, \dots$, where $A_i$ is the matrix formed by replacing the $i$-th column of $A$ by the column vector $B$. If $\tr A_i$, $\tr A$ are convergent series, then Cramer's formula holds for infinity case.

\medskip

In other hand, system of equations $AX=B$ ($A$, $X$, $B$ can have an infinite dimension) implies $X=A^{-1}B$.

\begin{Theorem}[Rouch\'e-Capelli Theorem (Kronecker-Capelli Theorem)]
	Let $m, n\in\mathbb{N}\cup\{\infty\}$.
	A system of $m$ linear equations in $n$ variables $Ax=b$ is compatible if and only if both the incomplete and complete matrices ($A$ and $[A|b]$ respectively) are characterised by the same $rank A=rank [A|b]$.  
\end{Theorem}

\begin{proof}
	Let $m$, $n\in\mathbb{N}\cup\{\infty\}$. 
	The system of linear equations $Ax=b$ can be interpreted as a linear mapping $f:F^n\to F^m$, by $f(x)=Ax$, such that $A\in M(n, R)$, $R$ be an integral domain.
	
	\medskip
	
	This system is determined if one solution exists, t.e. if there exists $x_0$ such that $f(x_0)=b$. This means that the system is determined if $b\in Im(f)$.
	
	\medskip
	
	The basis spanning the image vector space $(Im(f), +, \cdot)$ is composed of the column vectors of the matrix $A$:
	$$B_{Im(f)}=\{I_1, I_2, \dots, I_n\}, A=(I_1I_2\dots I_n).$$
	Thus, the fact that $b\in Im(f)$ is equivalent to the fact that $b$ belongs to the span of the column vectors of the matrix $A$:
	$$b=(I_1, I_2, I_3, \dots).$$
	This is equivalent to say that the rank of
	$$A=(I_1I_2I_3\dots)$$
	and 
	$$[A|b]=(I_1I_2I_3\dots b)$$
	have the same rank.
	Thus, the system is compatible if $rank A=rank [A|B]$. 
\end{proof}

\medskip

Let $B=\{v_1, v_2, v_3, \dots\}$, $B'=\{u_1, u_2, u_3, \dots\}$. Then for $i=1, 2, 3, \dots$ we compute coordinates $\alpha_1^{(i)}$, $\alpha_2^{(i)}$, $\alpha_3^{(i)}$, $\dots$ of the basis vector $B'$ in basis $B$:
$$u_i=\sum_{j=1}^{\infty}\alpha_j^{(i)}v_j.$$
Hence a transition matrix is of the form:
$$\left( \begin{array}{cccc}
\alpha_1^{(1)} & \alpha_1^{(2)} & \alpha_1^{(3)} & \ldots \\
\alpha_2^{(1)} & \alpha_2^{(2)} & \alpha_2^{(3)} & \ldots \\
\alpha_3^{(1)} & \alpha_3^{(2)} & \alpha_3^{(3)} & \ldots \\
\vdots & \vdots & \vdots & \ddots 
\end{array} \right).$$

\medskip

Let $L\colon U\to V$ be a linear transformation, where $U$, $V$ be a linear spaces such that $\dim U=m$, $\dim V=n$ ($m, n$ can be $\infty$) and a basis of $U$ be $\{u_1, u_2, \dots, u_m\}$, a basis of $V$ be $\{v_1, v_2, \dots, v_n\}$. For $i=1, 2, \dots, m$ and $j=1, 2, \dots, n$ compute 
$$L(u_i)=\sum_{j=1}^n\alpha_j^{(i)}v_j.$$
Then a transformation matrix of transformation $L$ is of the form:
$$\left( \begin{array}{cccc}
\alpha_1^{(1)} & \alpha_1^{(2)} & \ldots & \alpha_1^{(n)} \\
\alpha_2^{(1)} & \alpha_2^{(2)} & \ldots & \alpha_2^{(n)} \\
\vdots & \vdots & \ddots & \vdots \\
\alpha_m^{(1)} & \alpha_m^{(2)} & \ldots & \alpha_m^{(n)} 
\end{array} \right).$$

\medskip

We will try to find the eigenvalues and eigenvectors of the infinity matrix.

\medskip

Solve a characteristic equation:
$$\det (A-\lambda I)=0.$$
So we have to calculate
$$\det(A-\lambda I)=\exp(\tr(\log(A-\lambda I))),$$
where $\tr\log(A-\lambda I)$ be a convergent series.

\medskip

For an appropriate eigenvalue $\lambda$, we find the corresponding eigenvector $v=(x_1, x_2, x_3, \dots)$ from the system of equations:

$$(A-\lambda I)
\left( \begin{array}{ccc}
x_1 \\
x_2 \\
x_3\\
\vdots\\ 
\end{array} \right)
=\left( \begin{array}{ccc}
0 \\
0 \\
0\\
\vdots\\ 
\end{array} \right)$$

\medskip

Let $A$ be a matrix of any dimension, whose rows are given linearly independent vectors. We are building a block matrix $[AA^T\mid A]$. Applying elementary row operations we bring it to the block matrix of the form
$[G\mid A']$, where $G$ be the upper triangular matrix. The rows of $A'$ form orthogonal vectors.

.

\noindent Accepted: .....

\end{document}